\documentclass[12pt,a4paper,twoside,english]{smfart}
\usepackage{lmodern}
\usepackage[utf8]{inputenc}
\usepackage[francais]{babel}
\usepackage[T1]{fontenc}
\usepackage{amssymb}
\usepackage{amsmath}
\usepackage{amsthm}
\usepackage{amsfonts}
\usepackage{color}
\usepackage{enumerate}
\usepackage{graphicx}
\usepackage{url}

\usepackage[all,cmtip]{xy}
\usepackage{alltt}
\usepackage{smfthm}
\usepackage{footmisc}

\usepackage{fullpage}

\usepackage{hyperref}
\usepackage{epsfig}

\newtheorem{thm}{Theorem}
\newtheorem{propo}[thm]{Proposition}
\newtheorem{lem}[thm]{Lemma}
\newtheorem{cor}[thm]{Corollary}
\newtheorem{ex}[thm]{Example}
\newtheorem{rmk}[thm]{Remark}

\newcommand{\Q}{\mathbb{Q}}
\newcommand{\Z}{\mathbb{Z}}
\newcommand{\R}{\mathbb{R}}
\newcommand{\C}{\mathbb{C}}
\newcommand{\F}{\mathbb{F}}
\newcommand{\Hamil}{\mathbb{H}}
\newcommand{\disc}{\Delta}
\newcommand{\order}{\mathcal{O}}
\newcommand{\alg}{A}
\newcommand{\G}{\mathbb{G}}
\newcommand{\transp}[1]{\hspace{0.08em}{}^t\hspace{-0.08em}{#1}}
\def\PP{\mathbb{P}}

\def\Co{\mathcal{C}}

\def\B{{\rm B}}

\def\Alpha{{\mathcal A}}

\def\PP{{\mathbb P}}

\newcommand{\p}{\mathfrak{p}}
\newcommand{\Prm}{\mathfrak{P}}
\newcommand{\lieg}{\mathfrak{g}}
\newcommand{\liek}{\mathfrak{k}}
\newcommand{\liea}{\mathfrak{a}}
\newcommand{\liet}{\mathfrak{t}}
\newcommand{\aritgp}{\Gamma}
\newcommand{\mat}{\mathrm{M}}
\newcommand{\centre}{\mathrm{Z}}
\newcommand{\Id}{\mathrm{Id}}

\newcommand{\vertiii}[1]{{\vert\kern-0.25ex\vert\kern-0.25ex\vert #1 \vert\kern-0.25ex\vert\kern-0.25ex\vert}}

\def\N{{\rm N}}

\DeclareMathOperator{\GL}{GL}
\DeclareMathOperator{\SL}{SL}
\DeclareMathOperator{\liesl}{\mathfrak{sl}}
\DeclareMathOperator{\lieso}{\mathfrak{so}}
\DeclareMathOperator{\liesu}{\mathfrak{su}}
\DeclareMathOperator{\liesp}{\mathfrak{sp}}
\DeclareMathOperator{\SU}{SU}
\DeclareMathOperator{\U}{U}
\DeclareMathOperator{\SO}{SO}
\DeclareMathOperator{\nrd}{nrd}
\DeclareMathOperator{\trd}{trd}
\DeclareMathOperator{\tr}{tr}
\DeclareMathOperator{\rd}{rd}

\DeclareMathOperator{\diag}{diag}

\DeclareMathOperator{\Hom}{Hom}
\DeclareMathOperator{\Ad}{Ad}
\DeclareMathOperator{\rk}{rk}

\def\T{{\mathsf T}}
\def\n#1{{\Vert #1 \Vert_2}}
\def\V{{\mathbb V}}

\newtheorem*{TheoremSN}{Theorem}

\date{\today}

\author{Christian Maire}
\address{Institut FEMTO-ST, Univ. Bourgogne Franche-Comt\'e, 15B Av. des Montboucons, F-25030 Besan\c con c\'edex, France}
\email{christian.maire@univ-fcomte.fr}

\author{Aurel Page}
\address{INRIA, Univ. Bordeaux, CNRS, IMB, UMR 5251, F-33400 Talence, France}
\email{aurel.page@inria.fr}

\title{Codes from unit groups of division algebras over number fields}



\thanks{The authors would like to thank the Warwick Mathematics
Department and the Department of Mathematics at Cornell University for providing
a stimulating research atmosphere. We also thank Xavier Caruso for suggesting
the use of the sum-rank distance, and an anonymous referee for their valuable feedback.
We also thank the EPSRC for financial support via the EPSRC Programme Grant
EP/K034383/1 LMF: L-functions and modular forms.
CM was also partially supported by the Region Bourgogne Franche-Comté,
the ANR project FLAIR
(ANR-17-CE40-0012) and the EIPHI Graduate School (ANR-17-EURE-0002).}

\begin{abstract}
  Lenstra and Guruswami described number field analogues of the algebraic geometry
  codes of Goppa. Recently, the first author and Oggier generalised
  these constructions to other arithmetic groups: unit groups in number fields
  and orders in division algebras; they suggested to use unit groups in
  quaternion algebras but could not completely analyse the resulting codes.
  We prove that the noncommutative unit group construction yields asymptotically good families of
  codes for the sum-rank metric from division algebras of any degree, and we
  estimate the size of the alphabet in terms of the degree.
\end{abstract}

\begin{document}

\maketitle

\vspace{-1cm}

\section{Introduction}

Number field codes, introduced by Lenstra~\cite{Lenstra} and independently
rediscovered by Guruswami~\cite{Guruswami}, are number field analogues of the geometric codes
of Goppa~\cite{Goppa} built from curves over finite fields. In these
original constructions, the codes are constructed from the ring of
integers of a number field.

\smallskip

In~\cite{Maire-Oggier}, the first author and Oggier extended these ideas: they
explained how to construct codes from any arithmetic group, and they analysed
the parameters of the resulting codes in the case of the unit group of the ring
of integers of a number field, and in the case of an order in a division
algebra.
In every case, it was possible to obtain asymptotically good families of codes
using towers of number fields with bounded root discriminant.
The multiplicative group of an order in a
quaternion algebra was also considered with a partial analysis and the
question of constructing asymptotically good codes from these groups was left
open, short of having adapted techniques, as explained in Remark 12
of~\cite{Maire-Oggier}.

\smallskip

We completely analyse the noncommutative unit group codes
in any degree, by using a metric that is adapted to those groups.
We also measure the minimal distance of the code in terms of the sum-rank
distance introduced by Mart\'{i}nez-Pe\~{n}as~\cite{sumrank}, which also yields Hamming distance codes
by using the columns of the matrices as the alphabet.
We  analyse how the size of the alphabet varies with the degree of the
algebra, and we prove the following theorem.

\begin{TheoremSN}
  For all~$d\ge 2$, there exists a family of asymptotically good number field
  codes for the sum-rank distance,
  each obtained from the group of units of reduced norm~$1$ in a maximal order
  in a division algebra of degree~$d$,
  over a fixed
  alphabet~$\mat_d(\F_p)$, where~$\log p = c\log d + O(\log\log d)$
  and~$c>0$ is a constant.
\end{TheoremSN}

We did not try to get the sharpest possible bounds, but the
asymptotic of~$\log p$ in terms of~$d$ is probably correct apart from the value
of the constant~$c$.
Our analysis uses tools from the theory of arithmetic groups: the metric we 
choose is closely related to the canonical metric on the associated symmetric
space, and we use Macdonald's~\cite{Macdonald} and Prasad's~\cite{Prasad} volume
formulas. We also rely on a classical integration formula in $KAK$ coordinates
for the Haar measure of a semisimple Lie group (see for instance 
\cite[Proposition 5.28]{Knapp}), but the analysis of the dependence on the degree required computing
normalisation factors that we could not find in the literature.
Our method should generalise to arithmetic lattices in other semisimple groups.

\smallskip

One nice feature of the noncommutative case is that the formula
for the covolume of the arithmetic group is nicer than the one for the
regulator of a number field, since the relevant values of the zeta function
are evaluations at integers greater than~$1$.
For fixed~$d$, it is even possible to give closed formulas for the
parameters of the code using our techniques, and we carry out these computations
in the case~$d=2$. This is why we derive exact formulas whenever possible, and
then deduce an asymptotic analysis of the interesting quantities.

\smallskip

For the sake of comparison, we revisit the additive case and carry out the
analysis of the dependence on the degree. We obtain the following result:
although the additive and multiplicative groups have very different geometry,
the asymptotic behaviour of the code is the same.

\begin{TheoremSN}
  For all~$d\ge 2$, there exists a family of asymptotically good number field
  codes for the sum-rank distance, each obtained from the additive group of a
  maximal order in a division algebra of degree~$d$, over a fixed
  alphabet~$\mat_d(\F_p)$, where~$\log p = \frac{1}{2}\log d + O(\log\log d)$.
\end{TheoremSN}

In the context of coding theory, our results are mostly of theoretical value,
for two reasons. First, the code construction is not effective in its current
form (see Lemma~\ref{lem:boundratevol}, the element~$c\in G$ in the definition
of the code is not explicit); it seems like an interesting problem to derive an
effective version of our construction and an efficient decoding algorithm.
Second, the notion of asymptotically good
codes is rather coarse; it would be valuable to prove more precise estimates,
maybe in special cases in the spirit of our Section~\ref{sec:quaternion}, since
fixing the degree allows one to get explicit formulas for many relevant
parameters.
We hope that our results shed some extra light on the family of Goppa-style
constructions, and that they provide a starting ground for 
the study of related codes,
for instance over function fields, by changing the groups, by optimising the
shape of the ball ($\B$ in the notations of Section~\ref{sec:construction}), or
by considering polyalphabetic versions of our construction.

\medskip

The article is organised as follows. We first recall, in
Section~\ref{sec:construction}, the general construction of
arithmetic group codes following Maire--Oggier~\cite{Maire-Oggier}, and we review the basic properties
of central simple algebras over number fields in Section~\ref{sec:csa}.
In Section~\ref{section:Cartan}, we carry out the crucial volume computations
that we need to estimate the parameters of the codes. We analyse the multiplicative
construction in Section~\ref{sec:multconstr}, where we prove our main theorem
and give a detailed study of the quaternion case. Finally, we revisit the
additive construction in Section~\ref{sec:additiveconstr}.
 
\section{The construction}\label{sec:construction}

We recall the general construction as in \cite{Maire-Oggier}.
Given
\begin{enumerate}[(i)]
  \item a locally compact group~$G$ and a compact subset~$\B\subset G$,
  \item a lattice~$\aritgp\subset G$, i.e. a discrete subgroup with a
    fundamental domain of finite Haar measure,
  \item a map~$\Theta\colon \aritgp \to \Alpha^N$, where~$\Alpha$ is an alphabet
    (i.e. a finite set) and~$N\ge 1$ is an integer,
\end{enumerate}
we consider a code $\Co=\Theta(\aritgp\cap c\B)$, where~$c\in G$ is such
that~$|\aritgp\cap c\B|$ is maximal. There may be more than one such~$c$; we
simply pick an arbitrary~$c$. 
Codewords of  $\Co$ are elements of~$\Alpha^N$.
The main tool to estimate the rate of such codes is an idea
of Lenstra, which we express in the following lemma.

\begin{lem}\label{lem:boundratevol}
  Let~$\mu$ be a Haar measure on~$G$.
  If~$\Theta|_{\aritgp\cap c\B}$ is injective, then
  \[
    |\Co| \ge \frac{\mu(\B)}{\mu(G/\aritgp)}\cdot
  \]
\end{lem}
\begin{proof}
  This is~\cite[Lemma 1]{Maire-Oggier}.
\end{proof}

One natural way to construct such a code  from an arithmetic group is as follows. Given
 a number field~$F$, a linear algebraic group~$\G$ defined over~$F$
and an arithmetic group~$\aritgp\subset \G(\Z_F)$,
we can consider~$G = \G(F\otimes_\Q \R)$ in which~$\aritgp$,
under a mild extra condition on~$\G$,
is a lattice via the
natural embedding~$\aritgp \subset \G(F) \subset G$, by a theorem of
Borel and Harish-Chandra~\cite{BorelHC}. Then define
\[
  \Theta\colon \aritgp \to \prod_{\p\in S}\G(\Z_F/\p) \to \Alpha^N,
\]
where~$S$ is a finite set of prime ideals of~$F$ and~$\Alpha$ is related to the
groups~$\G(\Z_F/\p)$.
For instance, if for all~$\p \in S$ we
have an embedding~$\G(\Z_F/\p)\hookrightarrow \GL_d(\F_{q_0})$, then
we can take~$\Alpha = \mat_d(\F_{q_0})$, or~$\Alpha = \F_{q_0}^d$ in which case
the map~$\Theta$ picks the columns of the corresponding matrices.

We describe the interesting parameters of such codes.
The \emph{length} of a code~$\Co\subset \Alpha^N$ is~$N$, and we let~$q =
|\Alpha|$. The \emph{Hamming distance}~$d_H(x,y)$ between two
elements~$x,y\in\Alpha^N$ is the number of components in which they differ.
The \emph{rate} of~$\Co$ is ${\frac{\log_q|\Co|}{N}}$ with~$\log_q(t)=\log t/\log q$, and
the \emph{minimum Hamming distance} $d_H(\Co)$ of $\Co$ is the minimum Hamming
distance among all pairs of distinct codewords.

Families of codes $(\Co_i)_i$ with length $N_i \rightarrow \infty$ that satisfy
\[
\liminf_i \frac{\log_{q}|\Co_i|}{N_i} > 0,  {\rm and \ } \ \liminf_i \frac{d_H(\Co_i)}{N_i} > 0,
\]
are called {\em asymptotically good codes} (see for example~\cite{TVN} for a
good explanation of the notion in the context of algebraic geometry codes).

In the case of matrix alphabets, Mart\'{i}nez-Pe\~{n}as~\cite{sumrank}
introduced a finer distance: the sum-rank distance. Assume~$\Alpha = \mat_d(\F)$
where~$\F$ is a finite field;
if~$\Co\subset\Alpha^{N_0}$ is a code, then there is an attached
code~$\Co_{col}\subset\Alpha'^{N_0d}$ where~$\Alpha' = \F^d$, obtained by picking
the columns of the matrices.
The \emph{length} of the code~$\Co$ is defined to be~$N = N_0\cdot d$ and
we let~$q = |\F|^d$, so that they match the corresponding notions
for~$\Co_{col}$.
The \emph{sum-rank distance}~$d_R(x,y)$ between two elements~$x,y\in\Alpha^{N_0}$ is
\[
  d_R(x,y) = \sum_{i=1}^{N_0} \rk(x_i-y_i).
\]
Note that we have~$d_R(x,y) \le d_H(x,y)$ since the rank of a matrix is bounded
by the number of nonzero columns.
The \emph{rate} of~$\Co$ is ${\frac{\log_q|\Co|}{N}}$, and the \emph{minimum
sum-rank distance}~$d_R(\Co)$ of $\Co$ is the minimum sum-rank
distance among all pairs of distinct codewords.
Families of codes $(\Co_i)_i$ with length $N_i \rightarrow \infty$ that satisfy
\[
\liminf_i \frac{\log_{q}|\Co_i|}{N_i} > 0,  {\rm and \ } \ \liminf_i \frac{d_R(\Co_i)}{N_i} > 0,
\]
are called {\em asymptotically good codes for the sum-rank distance}.
If~$\Co_i$ is asymptotically good for the sum-rank distance,
then~$(\Co_i)_{col}$ is asymptotically good (for the Hamming distance).

In the arithmetic case considered above, the code length~$N$ is equal to~$d\cdot
|S|$, and~$q = q_0^d$.

The reason why we should expect arithmetic group codes to have large minimum distance
is the following. If two elements~$\gamma_1,\gamma_2 \in \aritgp$ correspond to
codewords that are close for~$d_R$ or~$d_H$, then they will be congruent modulo many
prime ideals~$\p$, and therefore~$\gamma_1^{-1}\gamma_2$ will be large in the
real embeddings, 
which will contradict the fact that they are bounded by the compact set~$\B$.

We are going to consider two instances of this construction: one where~$\G$ is the group of
units of reduced norm one of a division algebra, and one where~$\G$ is
the additive group of a division algebra. The precise specification of
the code is given in Section~\ref{sec:multconstr} for the
multiplicative case and Section~\ref{sec:additiveconstr} for the additive case.

\section{Central Simple Algebras: what we need}\label{sec:csa}

 Our main references  are \cite{PR} and \cite{Reiner}, however the literature
 is abundant. All our rings will be associative and
 unital, and all our algebras will be finite-dimensional.
 We write~$A^\times$ for the units of a ring~$A$.
 If~$R$ is a commutative ring and~$d\ge 1$ an integer, let~$\mat_d(R)$
 be the $d\times d$ matrix algebra with coefficients in $R$.

\subsection{Generalities}
Let~$\alg$ be an algebra over a field~$F$. The \emph{centre}
of~$A$ is~$\centre(\alg)=\{a\in \alg \mid \ xa=ax, \forall x \in \alg\}$. 
We say
that~$\alg$ is \emph{central} if~$\centre(\alg) = F$, that~$\alg$ is a
\emph{division algebra} if~$\alg^\times = \alg \setminus\{0\}$, and that~$\alg$
is \emph{simple} if the only two-sided ideals of~$\alg$ are~$\{0\}$ and~$\alg$.
Every division algebra is simple.
The dimension of a central simple algebra over~$F$ is a square~$d^2$, and~$d$ is
called the \emph{degree} of~$\alg$ over~$F$.
A central simple algebra of degree~$2$ is called a \emph{quaternion algebra}.

\smallskip

Let~$\alg$ be a central simple algebra of degree~$d$ over a field~$F$.
There exists a finite extension~$L/F$ and an isomorphism~$\alg\otimes_F L \cong
\mat_d(L)$ of algebras over~$L$; let~$\varphi$ be such an isomorphism.
For all~$x\in\alg$, the determinant~$\det\varphi(x)$ (resp. the
trace~$\tr\varphi(x)$) is in~$F$ and is independent
of~$L$ and~$\varphi$; it is called the \emph{reduced norm}~$\nrd(x)$ (resp.
the \emph{reduced trace}~$\trd(x)$) of~$x$.
The map~$\trd\colon\alg\to F$ is~$F$-linear, the map~$\nrd\colon \alg\to F$ is
multiplicative, and~$\alg^\times = \{x\in\alg\mid \nrd(x)\neq 0\}$.

\begin{ex} \label{exemple:Hamil} Let~$\Hamil$ be the algebra of Hamiltonian quaternions:
  \[
    \Hamil = \R + \R i + \R j + \R k\text{ with }i^2=j^2=k^2=ijk=-1.
  \]
  Then~$\Hamil$ is a central division algebra over~$\R$.
  For all~$x,y,z,t\in\R$ we have~$\nrd(x+yi+zj+tk) = x^2+y^2+z^2+t^2$
  and~$\trd(x+yi+zj+tk) = 2x$.
\end{ex}

\subsection{Over  number fields}
 Let $F$ be a number field of degree $n$ over $\Q$. Denote by~$\Z_F$ the ring of
 integers of $F$, and by $\PP_\infty$ the set of infinite places of $F$. If
 $I\subset \Z_F$ is a nonzero  ideal of $\Z_F$, let~$\N(I)=|\Z_F/I|$ be its 
 norm. Let~$\alg$ be a central simple algebra of degree~$d$ over~$F$.

\subsubsection{}
Let~$x\in\alg$. Then the \emph{absolute norm}~$\N(x)\in\Q$ of~$x$ is the absolute value
of the determinant of the matrix of left multiplication by~$x$ on~$\alg$, seen
as a matrix in~$\mat_{d^2n}(\Q)$. We have~$\N(x) = |N_{F/\Q}(\nrd(x))^d|$, where $N_{F/\Q}$ denotes the norm in $F/\Q$.
When $A$ is a division algebra, $\N(x)=0$ if and only if $x=0$.

\subsubsection{} 
An \emph{order} in $\alg$ is a subring~$\order\subset\alg$ that is
finitely generated as a~$\Z$-module and such that~$\order F = \alg$.
If~$\order$ is an order and~$x\in\order$ is such that~$\N(x)\neq 0$, then~$\N(x)
= |\order/x\order|$.

Let~$\order$ be a maximal order (i.e. not properly contained in a larger order),
and let~$\p$ be a prime ideal of~$\Z_F$; let $q=\N(\p)$. We say that
\begin{enumerate}[\quad $\bullet$]
  \item $\p$ is \emph{unramified} in~$\alg$
    if~$\order/\p\order\cong\mat_d(\F_q)$;
  \item $\p$ is \emph{ramified} in~$\alg$ otherwise.
\end{enumerate}
The number of primes that ramify in~$\alg$ is finite.

Let~$\p$ be a prime ideal of~$\Z_F$,
and denote by $F_{\p}$ the completion of $F$ at $\p$. Then there exists an isomorphism 
\[
\alg\otimes_F F_{\p} \cong \mat_{d_\p}(D_\p),
\]
where $D_\p$ is a central division algebra over $F_{\p}$; let us write
$\dim_{F_\p}D_\p=e_\p^2$. We then have $e_\p d_\p=d$, the prime~$\p$ is ramified if and
only if $e_\p >1$, and more generally we have~$\order/\Prm \cong
\mat_{d_\p}(\F_{q^{e_\p}})$ where~$\Prm$ is the unique maximal two-sided ideal
of~$\order$ containing~$\p$.

\smallskip

Let~$\sigma \in \PP_\infty$ be an infinite place of~$F$.
If~$\sigma$ is complex, then there is an isomorphism
\[
  \alg\otimes_F F_\sigma \cong \mat_d(\C),
\]
and in this case, we say that $\sigma$ is \emph{unramified}.
If~$\sigma$ is real, then there is an isomorphism
\begin{enumerate}[\quad $\bullet$]
  \item $\alg\otimes_F F_\sigma \cong \mat_d(\R)$, in which case we say
    that~$\sigma$ is \emph{unramified} in~$\alg$, or
  \item $\alg\otimes_F F_\sigma \cong \mat_{d/2}(\Hamil)$, in which case we say
    that~$\sigma$ is \emph{ramified} in~$\alg$.
\end{enumerate}

\smallskip

In all cases,  we fix an isomorphism extending~$\sigma$ as above, and still call
it~$\sigma\colon \alg\otimes_F F_\sigma \cong \mat_{d_\sigma}(D_\sigma)$,
where~$D_\sigma$ is a central division algebra of degree~$e_\sigma\in\{1,2\}$ over~$F_\sigma$,
and we write~$n_\sigma = [F_\sigma : \R]$.

\subsubsection{}
We denote by $\disc_F$ the
absolute value of the discriminant of $F$ and by $\rd_F$ its
\emph{root discriminant}~$\disc_F^{1/n}$.
Recall that if $K/F$ is an unramified extension of number fields then $\rd_K=\rd_F$.
The \emph{reduced
discriminant}~$\delta_A$ of $A$ is defined as
\[
  \delta_A = \prod_{\p} \p^{d(1-1/e_\p)} = \prod_{\p} \p^{d - d_\p},
\]
where the product is over all prime ideals of~$\Z_F$. The \emph{absolute
discriminant}~$\disc_A$ of $A$ is defined as
\[
\disc_A = \disc_F ^{d^2} \N(\delta_A)^d.
\]

\section{Volumes} \label{section:Cartan}

\subsection{Cartan decomposition}
Let~$D$ be a division algebra over~$\R$, so that~$D$ is isomorphic to one
of~$\R$, $\C$ or~$\Hamil$. 
We write~$F = \centre(D)$ the centre of $D$, $e$ the
degree of~$D$ over~$F$, and~$n = [F\colon\R]$.
Hence:  $e=n=1$ if $D=\R$;  $e=1$, $n=2$ if $D=\C$; and $e=2$, $n=1$ if $D=\Hamil$.
We will later apply the results of this section to the completion~$\mat_d(D)$ of a central 
simple algebra over a number field at a real or complex place~$\sigma$, and
the notations for those degrees will become~$e_\sigma, n_\sigma$ and~$d_\sigma$.

There exists a unique $\R$-linear and 
anti-multiplicative
involution~$x\mapsto \bar{x}$ on~$D$ such that for all~$x\in D^\times$ we
have~$x\bar{x}\in \R_{> 0}\subset D$: it is called the \emph{canonical
involution} of~$D$. Explicitly, the canonical involution is the identity
on~$\R$, the complex conjugation on~$\C$, and the quaternionic conjugation
on~$\Hamil$. In this section, we fix~$d\ge 1$.

\bigskip

Consider the semisimple Lie group
\[
  G = \SL_d(D) = \{g\in\mat_d(D)\mid \nrd(g)=1\},
\]

the following maximal compact subgroup of~$G$
\[
  K = \SU_d(D) = \{g\in G \mid g \transp{\bar{g}} = \Id\},
\]
and define
\[
  A = \left\{\begin{pmatrix}\exp(a_1) &  & 0 \\ & \ddots & \\ 0 & &
  \exp(a_d)\end{pmatrix} : a_1,\dots,a_d \in\R \mid \sum_{i=1}^d a_i = 0\right\}\subset G.
\]

\begin{ex}
  If~$D = \Hamil$ and~$d=2$, then
  \[
    \SL_d(D) = \SL_2(\Hamil) =
    \left\{ g \in \mat_2(\Hamil) \mid \nrd (g) = 1
    \right\},
  \]
  where for~$a,b,c,d\in\Hamil$ we have
  \[
    \nrd \begin{pmatrix}
        a & b \\
        c & d
    \end{pmatrix}
    =
    \left\{
      \begin{array}{ll}
        \nrd(ad)              & \text{if }c=0 \\
        \nrd(ac^{-1}dc - bc)  & \text{if }c\neq 0.
      \end{array}
    \right.
  \]
  Note that the set
  $
    \left\{
      \begin{pmatrix}
        a & b \\
        c & d
      \end{pmatrix}
      \in \mat_2(\Hamil)
      \mid ad-bc=1
    \right\}
  $
  is not stable under multiplication.

  In this case we have
  \[
    K = \SU_d(D) = \SU_2(\Hamil) =
    \left\{
      \begin{pmatrix}
        a & b \\
        c & d
      \end{pmatrix}
      \in \mat_2(\Hamil)
      \mid a\bar{a}+b\bar{b} = c\bar{c}+d\bar{d} = 1,\ a\bar{c}+b\bar{d}=0
    \right\}.
  \]
\end{ex}

Let  $\centre_K(A)=\{a\in A \mid \ ak=ka\ \forall k\in K\}$ be the centraliser
of $A$ in $K$. Explicitly, we have
\begin{enumerate}[\quad $\bullet$]
  \item if~$D=\R$, then~$\centre_K(A)\cong \{\pm 1\}^{d-1}$ is the group of
    diagonal matrices with coefficients~$\pm 1$ on the diagonal and
    determinant~$1$;
  \item if~$D=\C$, then~$\centre_K(A)\cong \U(1)^{d-1}$ is the group of diagonal
    matrices with coefficients in~$\C$ of absolute value~$1$ on the diagonal and
    determinant~$1$;
  \item if~$D=\Hamil$, then~$\centre_K(A) \cong \SL_1(\Hamil)^d \cong
    \SU_2(\C)^d$ is the group of diagonal matrices with coefficients in~$\Hamil$
    of reduced norm~$1$ on the diagonal.
\end{enumerate}

Let us recall the Cartan decomposition of $G$ that will be useful for the
computation of the different volumes (see for example \cite[Chapter V, \S 4]{Knapp}).

\begin{thm}[Cartan decomposition]\label{thm:cartan}
  Let~$G = \SL_d(D)$, $K$ and~$A$ be as above. Then we have~$G = KAK$.
  In the decomposition~$g=k_1ak_2$ of an element~$g\in G$ with~$k_i\in K$ and~$a
  = \diag(\exp(a_1),\dots,\exp(a_d))\in A$, the~$a_i$ are unique up to permutation.
  Moreover, let~$S\subset G$ be the subset of elements~$g\in G$ such that
  the~$a_i$ are distinct. Then for all~$g$ in~$S$, the pair~$(k_1,k_2)$ is
  uniquely determined up to changing~$(k_1,k_2)$ into~$(k_1z^{-1},zk_2)$ with~$z\in\centre_K(A)$.
\end{thm}

\begin{rmk}
  When~$D=\R$ or~$\C$, this is also known as the singular value decomposition.
  Note that~$G\setminus S$ has zero measure, so that we can and will restrict
  to~$S$ when computing integrals.
\end{rmk}

Define the function~$\rho\colon G \to \R_{\ge 0}$ such that for all~$g\in G$
with Cartan decomposition~$g = k_1 a k_2$, where~$a =
\diag(\exp(a_1),\dots,\exp(a_d))$, we have
\[
  \rho(g) = \max_i |a_i|.
\]
This is well-defined since the~$a_i$ are unique up to permutation.
Note that we have~$\rho(g^{-1}) = \rho(g)$ for all~$g\in G$.

Let us give a series of corollaries of Theorem~\ref{thm:cartan} that will be
useful to estimate the minimal Hamming distance of the multiplicative codes. 

Let~$\|\cdot\|_2 \colon D^d \to \R$ be the norm on~$D^d$ defined by~$\|x\|_2 = \left(\sum_{i=1}^d x_i\bar{x}_i
\right)^{1/2}$ for all~$x\in D^d$, and let~$\vertiii{\cdot}\colon \mat_d(D) \to
\R_{\ge 0}$ be the corresponding operator norm, that is:

$$\vertiii{g}=\sup_{x\neq 0}
\displaystyle{\frac{\|g\cdot x\|_2}{\|x\|_2}} \text{ for all }g \in
\mat_d(D).$$

\begin{cor}\label{cor:opnorm}
  For all~$g\in G$, we have
  \[
    \rho(g) = \log\max(\vertiii{g}, \vertiii{g^{-1}}).
  \]
\end{cor}
\begin{proof}
 As $\| \cdot \|_2$ is bi-invariant by $K$, in Cartan decomposition (Theorem~\ref{thm:cartan}) we have~$\vertiii{g} = \max_i
  \exp(a_i)$. Applying this to~$g$ and~$g^{-1}$ gives the result.
\end{proof}

\medskip

\begin{cor}\label{cor:subadd}
  For all~$g,h\in G$, we have~$\rho(gh)\le \rho(g)+\rho(h)$.
\end{cor}
\begin{proof}
  This follows from Corollary~\ref{cor:opnorm} and the submultiplicativity of operator norms.
\end{proof}

\begin{cor}\label{cor:nrdbound}
  For all~$g\in G$, we have
  \[
    |\nrd(g-1)| \le 2^d\exp(d\rho(g)).
  \]
\end{cor}
\begin{proof}
  Let~$x=g-1$.
  We claim that we have~$|\nrd(x)| \le \vertiii{x}^d$. To see this, let~$V=D^d$
  viewed as an~$\R$-vector space of dimension~$de^2n$.
  The absolute value of the determinant of~$x$ viewed as an endomorphism of~$V$
  is~$|\nrd(x)|^{ne^2}$.
  So multiplication by~$x$ scales volumes by~$|\nrd(x)|^{ne^2}$; applying this
  to a ball for the norm~$\|\cdot\|_2$, noting that the volume
  of the ball of radius~$R$ is proportional to~$R^{de^2n}$ and using the definition of
  the operator norm proves the claim.
  
  On the other hand we have
  \[
    \vertiii{x} \le \vertiii{g} +1 \le 2\max(1,\vertiii{g}) \le 2\exp(\rho(g)),
  \]
  where the last inequality follows from Corollary~\ref{cor:opnorm}.
\end{proof}

\subsection{Haar measure on $\SL_d(D)$} \label{section:Haar-measure}
If~$M,N\in\mat_d(D)$, we write~$[M,N]$ for the Lie bracket~$MN-NM$, and
if~$g\in\GL_d(D)$, we write~$\Ad(g)M = gMg^{-1}$.
The Lie algebra of~$G = \SL_d(D)$ is~$\lieg = \liesl_d(D) = \{X\in
\mat_d(D) \mid \trd(X)=0\}$. We equip the~$\R$-vector space~$\liesl_d(D)$ with
the positive definite inner product~$(X,Y)\mapsto \tr_{\centre(D)/\R}\trd(\transp{\overline{X}}Y)$, with
corresponding norm~$\|X\|^2 = n\trd(\transp{\overline{X}}X)$.
This gives~$G=\SL_d(D)$ the
structure of a Riemannian manifold with a metric on~$G$ that is invariant under
left translations by arbitrary elements of~$G$ and under right translation by
elements of~$K$. In particular, this defines a volume form $d\mu$  on~$G$,
on~$K$ and on~$\centre_K(A)$.

Let us start with the computation of the volume of $\centre_K(A)$ with respect
to~$d\mu$.

 \begin{lem} \label{lemm:volcentre} Let~$G = \SL_d(D)$ and~$K,A$ be as above. Then $\mu(\centre_K(A))$ equals
\begin{enumerate}[\quad $\bullet$]
\item $2^{d-1}$ when $D=\R$,
\item $(2\sqrt{2}\pi)^{d-1}\sqrt{d}$  when $D=\C$,
\item $(4\sqrt{2}\pi^2)^d$ when $D=\Hamil$.
\end{enumerate}
\end{lem}

\begin{proof}~
\begin{enumerate}[\quad $\bullet$]
  \item If~$D=\R$, the group~$\centre_K(A) \cong \{\pm 1\}^{d-1}$ is finite and the corresponding
    measure is the counting measure.
  \item If~$D=\C$, the group~$\centre_K(A)$ is a product of
    circles~$U(1)^{d-1}$, which we parametrise as the image of~$[0,2\pi]^{d-1}$
    under the map
    \[
      (\theta_1,\dots,\theta_{d-1}) \mapsto \diag\left(\exp(i\theta_1),
      \dots, \exp(i\theta_{d-1}), \exp\Bigl(-i\sum_{k=1}^{d-1}\theta_k\Bigr)\right).
    \]
    The Gram matrix of the corresponding tangent vectors is
    the~$(d-1)\times(d-1)$ matrix
    \[
      \begin{pmatrix}
        4      & 2      & \dots  & 2      \\
        2      & 4      & \ddots & \vdots \\
        \vdots & \ddots & \ddots & 2      \\
        2      & \dots  &      2 & 4
      \end{pmatrix},
    \]
    which has determinant~$2^{d-1}d$. So the corresponding volume
    is~$(2\pi)^{d-1}\sqrt{2^{d-1}d}$.
  \item If~$D=\Hamil$, the group~$\centre_K(A)$, as a Riemannian manifold, is a
    direct product of $d$ copies of~$\SL_1(\Hamil) = \{x+yi+zj+tk\mid
    x^2+y^2+z^2+t^2 = 1\}$, which is a $3$-dimensional sphere
    of radius~$\sqrt{2}$ with our normalisation. It is well-known that the
    volume of a $3$-dimensional sphere of radius~$R$ is~$2\pi^2 R^3$, giving the
    formula.
\end{enumerate}
\end{proof}

More generally, we can compute volumes using the following formula:
\begin{propo}\label{prop:integralKAK}
  Let~$G = \SL_d(D)$ and~$K,A$ be as above, and let~$f\in L^1(G)$. Then the
  integral~$\displaystyle{\int_G f \ d\mu}$ equals
  \[
    \frac{(ne)^{\frac{d-1}{2}}\sqrt{d}}{\mu(\centre_K(A))}\int_{K\times K}
    \int_{a_i}
    \prod_{1\le i < j \le d}\sinh(a_i-a_{j})^{ne^2} f(k_1
    a
    k_2)da_idk_1dk_2,
  \]
  where:
  \begin{enumerate}[\quad $\bullet$]
    \item $a = \diag(\exp(a_1),\dots,\exp(a_d))$, and
    \item the inner integral is over the set
      \[
        \left\{(a_1,\dots,a_{d-1})\in\R^{d-1} \mid a_1 > a_2 > \dots > a_{d-1} >
        -\sum_{i=1}^{d-1}a_i \right\}
      \]
      and~$a_d = -\sum_{i=1}^{d-1}a_i$.
  \end{enumerate}
\end{propo}
\begin{proof}
  We will obtain the formula by pulling back the metric along the map
  \[
    \Psi \colon K\times A\times K \to G,
  \]
  and computing the pullback by using the decomposition of~$\liesl_d(D)$
  according to restricted roots of~$A$, as in \cite{Hu-Yan}, keeping track of
  all constants. We write~$\liea$ (resp.~$\liek$) for the Lie algebra of~$A$
  (resp. of~$K$).
  
  First we compute the differential of~$\Psi$. Let~$x = (k_1,a,k_2)\in K\times
  A\times K$, where~$a = \diag(\exp(a_i))$. Using the canonical isomorphism
  between the tangent space of a Lie group at an arbitrary point and its Lie
  algebra, we obtain that the differential of~$\Psi$ at~$(k_1,a,k_2)$ is the
  map~$d\Psi_x \colon \liek \times \liea \times \liek \to \lieg$ that sends
  \[
    (X_1,Y,X_2) \mapsto \Ad(k_2^{-1})\Ad(a^{-1})X_1 + Y + X_2.
  \]
  To compute the pull-back of the volume form to~$K\times A\times K$, it is
  enough to compute~$\Delta_x = \det \bigl( d\Psi_x \transp{(d\Psi_x)} \bigr)^{1/2}$.
  We compute this determinant by blocks. The term in~$\liea$ does not
  contribute. We compute the other terms by decomposing~$\mat_d(D)$ in $2\times
  2$ blocks corresponding to rows and columns~$i$ and~$j$.
  Consider the matrices
  \[
    F_+ = \begin{pmatrix}0 & 1\\1 & 0\end{pmatrix} \text{ and } F_- =
      \begin{pmatrix}0 & 1\\-1 & 0\end{pmatrix}\in \mat_2(D).
  \]
  We have~$\Ad(a)F_+ = \cosh(a_i-a_j)F_+ + \sinh(a_i-a_j)F_-$ and~$\Ad(a)F_- =
  \cosh(a_i-a_j)F_- + \sinh(a_i-a_j)F_+$. In addition, let~$D^0 = \{z\in D\mid
  \bar{z}+z=0\}$. Then~$\liek$ is the sum over all the blocks of~$F_-\R\oplus
  F_+D^0$, and of~$Z_\liek(\liea)$, and those are orthogonal to~$F_+\R$ and~$F_-D^0$.
  Using this, we see that the matrix of~$d\Psi_x$ is a direct sum of blocks of
  the form
  \[
    \begin{pmatrix}
      \sinh(a_i-a_j) & 0\\
      \cosh(a_i-a_j) & 1
    \end{pmatrix}
  \]
  each appearing with multiplicity~$\dim_\R D = ne^2$, and of
  the block~$(1\ 1)$ appearing with multiplicity~$(e^2-1)d$.
  We therefore have
  \[
    \Delta_x = 2^{\frac{(e^2-1)d}{2}}\prod_{1\le i<j\le d}\sinh(a_i-a_j)^{ne^2}.
  \]
  The map~$\Psi$ is not injective, and by Theorem~\ref{thm:cartan}, outside of a
  set of zero measure, the fibre above~$k_1ak_2$ is~$\{(k_1z,a,z^{-1}k_2) :
  z\in Z_K(A)\}$. With the same computation we obtain that the volume of this
  fibre is~$2^{\frac{(e^2-1)d}{2}}\mu(Z_K(A))$. This gives
  \[
    \int_G f \ d\mu = 
    \frac{1}{\mu(\centre_K(A))}\int_{K\times K}
    \int_{a\in A^+}
    \prod_{1\le i < j \le d}\sinh(a_i-a_{j})^{ne^2} f\left(k_1 a k_2\right)dadk_1dk_2,
  \]
  where~$A^+\subset A$ is the subset of diagonal matrices with decreasing entries.
  We parametrise
  \[
    A^+ = \left\{\diag(\exp(a_i)) : (a_1,\dots,a_{d-1})\in\R^{d-1} \mid a_1 > a_2 > \dots > a_{d}
         \text{ and } a_d = -\sum_{i=1}^{d-1}a_i\right\},
  \]
  and the factor corresponding to this second change of variables is the
  square root of the determinant of the Gram matrix of the
  tangent vectors~$\diag(0,\dots,1,0,\dots,0,-1)\in\lieg$. This Gram matrix is
  the~$(d-1)\times(d-1)$ matrix
  \[
    ne
    \begin{pmatrix}
      2      & 1      & \dots  & 1      \\
      1      & 2      & \ddots & \vdots \\
      \vdots & \ddots & \ddots & 1      \\
      1      & \dots  &      1 & 2
    \end{pmatrix},
  \]
  and it has determinant~$(ne)^{d-1}d$.
\end{proof}


\begin{propo}\label{prop:volK}
  Let~$K$ be as above. Then we have
  \[
    \mu(K) = \kappa \prod_{k=1}^r \frac{\pi^{m_k+1}}{m_k!},
  \]
  where
      \begin{enumerate}[\quad (i)]
    \item if~$D=\R$:
      \begin{enumerate}[$\bullet$]
        \item if~$d$ is even, then~$r=d/2$, $\kappa = 2^{d^2/2-d/4}$ and
      \[m_k = 2k-1\text{ for }k\le r-1 \text{ and }m_r = r-1;\]
        \item if~$d$ is odd, then~$r=(d-1)/2$, $\kappa = 2^{d^2/2+d/4-3/4}$ and
      \[m_k = 2k-1 \text{ for all }k\le r;\]
      \end{enumerate}
    \item\label{item:volsud} if~$D=\C$, then~$r=d-1$, $\kappa = 2^{d^2+d/2-3/2}\sqrt{d}$ and
      \[m_k = k \text{ for all }k\le r;\]
    \item if~$D=\Hamil$, then~$r=d$, $\kappa = 2^{2d^2+d/2}$ and
      \[m_k = 2k-1 \text{ for all }k\le r.\]
  \end{enumerate}
\end{propo}
\begin{proof}
  We apply Macdonald's formula \cite{Macdonald}. 
  Recall (see \cite[Part II, \S 19]{Bump}) that the roots of~$\liek$  with respect to~$\liet$  are the nonzero
  morphisms~$\alpha\in \Hom_\R(\liet,\C)$ such that there exists a nonzero~$X_\alpha\in
  \liek\otimes\C$ such that~$[t,X_\alpha]=\alpha(t)X_\alpha$ for all~$t\in\liet$;
  for each root~$\alpha$, the attached coroot~$\alpha^\vee
  \in \liet\otimes\C$ is the unique element~$\alpha^\vee\in
  \C \cdot [X_\alpha,X_{-\alpha}]$
  such that~$[\alpha^\vee,X_\alpha] = 2X_\alpha$ and~$[\alpha^\vee,X_{-\alpha}]
  = -2X_{-\alpha}$.
  In each case, we give the Lie
  algebra~$\liek$ of~$K$, and in Macdonald's notations, $\liet\subset\liek$,
  $\liet_\Z\subset\liet$, and the list of roots~$\alpha$, their corresponding
  root vectors~$X_\alpha$ and coroots~$\alpha^\vee\in\liek\otimes \C$.
  The list of~$m_k$ is standard (see \cite[\S 1.5]{Prasad} or \cite[Chap. VIII, \S 13 (VI)]{Bourbaki}).
  
  \smallskip
  
    \begin{enumerate}[\quad (i)]

  \item When~$D=\R$, $K = \SO_d(\R)$ and~$\liek = \lieso_d(\R)$ is the Lie algebra of
    antisymmetric matrices (which all have trace~$0$), and we
    have~$\liek_\C\cong \lieso_d(\C)$. We choose~$\liet\subset\liek$ to be the space of
    matrices that are block-diagonal with~$2\times 2$ blocks of the
    form~$\begin{pmatrix}0 & \theta_k \\ -\theta_k & 0\end{pmatrix}$ for~$k = 1, \dots, r$.
    Then~$\liet_\Z\subset\liet$ is the lattice
    of elements with~$\theta_k\in\Z$ for all~$k$.
    We will also write~$\theta_k$ the corresponding linear form on~$\liet$.
    For~$k<\ell\le r$ and~$M\in \mat_2(\R)$, let~$R_{k,\ell}(M)$ be the block
    matrix~$\begin{pmatrix}0 & M \\ -\transp{M} & 0\end{pmatrix}$ embedded in
    the~$(2k-1,2k,2\ell-1,2\ell)$-th block of a~$d\times d$ matrix.
    For~$k\le r$ and~$v\in\R^2$ a column vector, let~$R_k(v)$ be the block
    matrix~$\begin{pmatrix}0 & v\\ -\transp{v} & 0\end{pmatrix}$, embedded in
    the~$(2k-1,2k,d)$-th block of a~$d\times d$ matrix.
    For~$k\le r$, let~$F_k\in\liet$ be
    the matrix with~$\theta_k=1$ and all other coefficients~$0$.
    Then the roots of~$\liek$ with respect to~$\liet$ are:
    \begin{enumerate}
      \item For~$1\le k < \ell\le r$, the~$\alpha = \pm (\theta_k-\theta_\ell)i$,
        with corresponding
        \[
          X_{\alpha} = R_{k,\ell}\!\begin{pmatrix}1 & 0\\0 & 1\end{pmatrix}
          \pm R_{k,\ell}\!\begin{pmatrix}0 & 1 \\ -1 & 0\end{pmatrix}\otimes i
          \in \liek\otimes\C
        \]
        and coroots~$\alpha^\vee = \pm(F_k-F_\ell)\otimes i\in \liet\otimes\C$.
      \item For~$1\le k < \ell\le r$, the~$\alpha = \pm (\theta_k+\theta_\ell)i$,
        with corresponding
        \[
          X_{\alpha} = R_{k,\ell}\!\begin{pmatrix}1 & 0\\0 & -1\end{pmatrix}
          \pm R_{k,\ell}\!\begin{pmatrix}0 & 1 \\ 1 & 0\end{pmatrix}\otimes i
        \]
        and coroots~$\alpha^\vee = \pm(F_k+F_\ell)\otimes i$.
      \item If~$d$ is odd, for~$1\le k \le r$, the~$\alpha = \pm \theta_k i$,
        with corresponding
        \[
          X_{\alpha} = R_{k}\!\begin{pmatrix}1 \\0 \end{pmatrix}
          \pm R_{k}\!\begin{pmatrix}0 \\ 1\end{pmatrix}\otimes i
        \]
        and coroots~$\alpha^\vee = \pm 2F_k\otimes i$.
    \end{enumerate}
   
  \item When~$D=\C$, $K = \SU_d(\C)$ and~$\liek = \liesu_d(\C)$ is the Lie algebra of
    anti-Hermitian matrices with trace~$0$, and we have~$\liek\otimes\C\cong
    \liesl_d(\C)$. We choose~$\liet\subset\liek$ to be
    the space of diagonal matrices with coefficients~$i\theta_k$
    where~$\theta_k\in\R$ and~$\sum_{k=1}^d\theta_k=0$.
    Then~$\liet_\Z\subset\liet$ is the lattice of elements with~$\theta_k\in\Z$
    for all~$k$.
    We will again write~$\theta_k$ the corresponding linear form on~$\liet$.
    For~$k<\ell\le r$ and~$z\in\C$, let~$R_{k,\ell}(z)$ be the
    matrix~$\begin{pmatrix}0 & z \\ -\bar{z} & 0\end{pmatrix}$, embedded in
    the~$(k,\ell)$-th block of a~$d\times d$ matrix.
    For~$k\le d$, let~$F_k$ be the matrix with~$\theta_k=1$ and all
    other coefficients~$0$.
    Then the roots of~$\liek$ are, for~$1\le k<\ell\le d$, the~$\alpha =
    \pm(\theta_k-\theta_\ell)i$, with corresponding
    \[
      X_{\alpha} = R_{k,\ell}(1) \pm R_{k,\ell}(i)\otimes i
    \]
    and coroots~$\alpha^\vee = \pm(F_k-F_\ell)\otimes i$.

  \item When~$D=\Hamil$, $K = \SU_d(\Hamil)$ and~$\liek = \liesu_d(\Hamil)$ is the Lie
    algebra of quaternion-anti-Hermitian matrices (which automatically have
    reduced trace~$0$), and we have~$\liek\otimes\C \cong \liesp_{2d}(\C)$. 
    We choose~$\liet\subset\liek$ to be the space of diagonal
    matrices with coefficients~$i\theta_k$ where~$\theta_k\in\R$.
    Then~$\liet_\Z\subset\liet$ is the lattice of elements with~$\theta_k\in\Z$
    for all~$k$.
    We will again write~$\theta_k$ the corresponding linear form on~$\liet$.
    For~$k<\ell\le r$ and~$w\in\Hamil$, let~$R_{k,\ell}(w)$ be the
    matrix~$\begin{pmatrix}0 & w \\ -\bar{w} & 0\end{pmatrix}$, embedded in
    the~$(k,\ell)$-th block of a~$d\times d$ matrix.
    For~$k\le r$, let~$F_k$ be the matrix with~$\theta_k=1$ and all
    other coefficients~$0$.
    Then the roots of~$\liek$ are
    \begin{enumerate}
      \item For~$1\le k<\ell\le d$, the~$\alpha = \pm(\theta_k-\theta_\ell)i$,
        with corresponding
        \[
          X_\alpha = R_{k,\ell}(1)\pm R_{k,\ell}(i)\otimes i
        \]
        and coroots~$\alpha^\vee = \pm(F_k-F_\ell)\otimes i$.
      \item For~$1\le k<\ell\le d$, the~$\alpha = \pm(\theta_k+\theta_\ell)i$,
        with corresponding
        \[
          X_\alpha = R_{k,\ell}(j)\pm R_{k,\ell}(ij)\otimes i
        \]
        and coroots~$\alpha^\vee = \pm(F_k+F_\ell)\otimes i$.
      \item For~$1\le k\le d$, the~$\alpha = \pm 2\theta_k i$,
        with corresponding
        \[
          X_\alpha = jF_k \pm ijF_k \otimes i
        \]
        and coroots~$\alpha^\vee = \pm F_k\otimes i$.
    \end{enumerate}
    
\end{enumerate}

  Computing~$\mu(\liet/\liet_\Z)$ and the norms of the coroots in each case
  gives the result from Macdonald's formula: in his notation we have
  \[
    \lambda = \mu(\liet/\liet_\Z)\prod_{\alpha} \|\alpha^\vee\|,
  \]
  and we finally let~$\kappa = 2^r\lambda$.
\end{proof}

\begin{cor} \label{cor:vol-compact}
  Let~$K$ be as above. We have
  \[
    \log\mu(K) = -\frac{n}{4}(ed)^2\log d + O(d^2).
  \]
\end{cor}
\begin{proof}
  Use Proposition~\ref{prop:volK} and~$\sum_{k=1}^r k\log k = \frac{r^2}{2}\log
  r + O(r^2)$.
\end{proof}

\subsection{Volume of a ball}
In order to find a lower bound for the volume of certain balls~$\B(t)$, we will need
to compute a lower bound for an integral of the form
  \[
    \int \prod_{i<j}\sinh(a_i-a_j)^{m}da_i,
  \]
where the integral is over the~$(a_i)_i$ with~$|a_i|\le t$ and~$\displaystyle{\sum_i a_i = 0}$. The domain for the~$a_i$ is the
intersection of a hypercube with the hyperplane~$\displaystyle{\sum a_i=0}$, i.e. a simplex. On
the other hand, the integrand is small when some~$|a_i-a_j|$ is small, that is
when~$(a_i)$ is close to one of the hyperplanes~$a_i=a_j$. To find a lower
bound, we will restrict to a subset of the simplex that is far from those
hyperplanes, and where the $a_i$ vary independently, so that we can compute the
integral. Moreover, the integrand increases exponentially, so the size of the subset does
not contribute significantly to the value of the integral, while the values of
the integrand do; so we need a subset where many of the~$a_i$ are close to~$t$.
This is achieved using the following technical lemma.

\begin{lem}\label{lem:intervals}
  Let~$k\ge 1$ be an integer. Then there exists~$k$
  intervals~$[\alpha_i,\beta_i]$ such that for all~$a_i\in [\alpha_i,\beta_i]$
  we have:
  \begin{enumerate}[(1)]
    \item $|a_i|\le 1$;
    \item $|\sum_{j=1}^k a_j|\le 1$;
    \item\label{item:longueur} $\beta_i-\alpha_i \ge \frac{1}{4(k+1)^2}$;
    \item $|a_i-a_j|\ge \frac{1}{4(k+1)^2}$;
    \item $|a_i+\sum_{j=1}^{k}a_j| \ge \frac{1}{4(k+1)^2}$;
    \item\label{item:many} $|\{j\colon a_j \ge \frac{1}{4}\}| \ge \frac{k+1}{5}$.
  \end{enumerate}
\end{lem}
\begin{proof}
  Let~$c_1,\dots,c_k$ be defined by~$c_i = \frac{2i}{k+1}-1$, except when~$k$
  is odd, where~$c_{\frac{k+1}{2}} = \frac{1}{k+1}$. Let~$\alpha_i =
  c_i-\frac{1}{8(k+1)^2}$ and~$\beta_i = c_i+\frac{1}{8(k+1)^2}$. We claim that they
  satisfy the required properties.
  \begin{enumerate}[(1)]
    \item In absolute value, the extremal points of the intervals
      are~$\frac{k-1}{k+1}+\frac{1}{8(k+1)^2}\le 1$.
    \item The sum~$\sum_{j=1}^k c_j$ is~$0$ if~$k$ is even,
      and~$\frac{1}{k+1}$ if~$k$ is odd. The required inequalities
      become~$|\frac{k}{8(k+1)^2}|\le 1$ and~$|\frac{1}{k+1} + \frac{k}{8(k+1)^2}|\le
      1$, which are true.
    \item By definition we have~$\beta_i-\alpha_i = \frac{1}{4(k+1)^2}$.
    \item The minimum separation between the centres~$c_i$ is~$\frac{1}{k+1}$,
      and we have~$\frac{1}{k+1}-\frac{1}{4(k+1)^2}\ge \frac{1}{4(k+1)^2}$.
    \item The minimum separation between a~$c_i$ and~$-\sum_{j=1}^k c_j$
      is at least~$\frac{1}{k+1}$, and we
      have~$\frac{1}{k+1}-\frac{1}{8(k+1)^2}-\frac{k}{8(k+1)^2}\ge
      \frac{1}{4(k+1)^2}$.
    \item If~$k\le 4$ it is obvious, and when $k\ge 5$ we have~$\alpha_i\ge \frac{1}{4}$ if and only if~$i\ge
      \frac{k+1}{8}(5+\frac{1}{2(k+1)^2})$, and the number of such~$i$ is at
      least~$\frac{k+1}{5}$.
  \end{enumerate}
\end{proof}

We can now give a lower bound for the volume of a certain ball $\B(t) \subset G$  with  the above Haar measure $d\mu$.

\begin{propo}\label{prop:volball}
  Let~$t\ge 1$, and let~$\B(t) = \{g\in G\mid \rho(g)\le t\}$. Assume~$d\ge 2$. Then 
  
  \[ \log \mu(\B(t)) \geq -\frac{3n(de)^2}{2} \log d + \frac{n(de)^2}{200} t + O(d^2)\cdot
  \]
\end{propo}

\begin{proof}
  We apply
  Proposition~\ref{prop:integralKAK} with~$f$ the indicator function of~$\B(t)$.
  The formula reads
  \[
  \displaystyle{  \mu(\B(t)) =
    \frac{(ne)^{\frac{d-1}{2}} \sqrt{d}}{\mu(\centre_K(A))} \mu(K)^2 I\text{, where }I
    =\int_{a_i}\prod_{i<j}\sinh(a_i-a_j)^{ne^2}da_i,}
  \]
  where the integral is over the~$a_i\in [-t,t]$ with~$(a_i)_i$ decreasing and~$\sum_i a_i=0$.
   To compute a lower bound for the integral, we apply
  Lemma~\ref{lem:intervals} to~$k=d-1$. We have
  intervals~$[\alpha_i,\beta_i]$, and after reordering the intervals and scaling
  them by~$t$, we obtain that
  \[
    I\ge \int\prod_{i<j}\sinh(a_i-a_j)^{ne^2}da_i,
  \]
  where for~$i<d$, $a_i$ ranges over some
  interval~$[t\alpha_{\sigma(i)},t\beta_{\sigma(i)}]$, and~$\displaystyle{a_d =
  -\sum_{i=1}^{d-1} a_i}$.
  Since ${x\mapsto \sinh(x)\exp(-x) = \frac{1 - \exp(-2x)}{2}}$ is increasing, for
  all $\displaystyle{x\ge \frac{t}{4d^2}}$ we have $${\sinh(x) \ge
  \frac{1-\exp(-t/2d^2)}{2}\exp(x) \ge \frac{1-\exp(-1/2d^2)}{2}\exp(x) \ge
  \frac{1}{4d^2}\exp(x)}\cdot$$ We get
  \[
    I \ge
    \left(\frac{1}{4d^2}\right)^{ne^2d(d-1)/2}\int
    \prod_{i<j}\exp(a_i-a_j)^{ne^2}da_i =
    \left(\frac{1}{4d^2}\right)^{ne^2d(d-1)/2}\int
    \exp\left(\sum_{i<j}(a_i-a_j)\right)^{ne^2}da_i.
  \]
  We compute the term~$\beta$ that appears in the exponential.
  For all~$(a_i)$ such that~$\displaystyle{\sum_{i=1}^d a_i=0}$, we have $$\begin{array}{rcl} \beta &=&\displaystyle{ \sum_{1\leq i<j\le
  d}(a_i-a_j) = \sum_{i<j}a_i - \sum_{i<j}a_j} \\ &=&\displaystyle{ \sum_i(d-i)a_i - \sum_j(j-1)a_j
  = \sum_{i=1}^d(d+1-2i)a_i }\\ &=& \displaystyle{-2\sum_{i=1}^d ia_i \cdot} \end{array}$$ Now since~$\displaystyle{a_d =
  -\sum_{i=1}^{d-1}a_i}$, we have~$\displaystyle{\beta = 2\sum_{i=1}^{d-1}(d-i)a_i}$. This gives
  \[
    I \ge (2d)^{-d(d-1)ne^2}\int \exp\Bigl(2ne^2\sum_{i=1}^{d-1}(d-i)a_i\Bigr)da_i  \cdot
  \]
  By properties~(\ref{item:longueur}) and~(\ref{item:many}) of Lemma~\ref{lem:intervals}, we obtain
  \[
    I \ge
    (2d)^{-d(d-1)ne^2}\exp\left(2ne^2\sum_{i=1}^{\lfloor d/5 \rfloor}i\frac{t}{4}\right)\left(\frac{1}{4d^2}\right)^{(d-1)ne^2}
    \geq (2d)^{-(d-1)(d+2)ne^2}\exp\left(\frac{d^2ne^2}{200}t\right) \cdot
  \]
  In particular, 
  \[
  \log I \ge -n(de)^2\log d + \frac{n(de)^2}{200} t + O(d^2) \cdot
  \]
  We conclude by using Corollary~\ref{cor:vol-compact} and Lemma~\ref{lemm:volcentre}. 
\end{proof}

\section{Multiplicative construction}\label{sec:multconstr}

We consider the following arithmetic group code.   Let~$F$ be a number field of degree $n$ over $\Q$, and let
$\alg$ be  a central division algebra of degree~$d\ge 2$ over~$F$.
If~$d=2$, assume that~$\alg$ is unramified at all real places.
Let $\order$ be a maximal order
in~$\alg$. We let~$\G$ be the algebraic group defined by the reduced norm~$1$
subgroup~$\alg^1\subset\alg^\times$, and
\[
  \aritgp = \order^1=\{x\in \order \mid \nrd(x)=1\}.
\]

Let~$S$ be a set of prime ideals of~$\Z_F$ that are unramified in~$\alg$  and such
that for all~$\p\in S$, the residue field~$\Z_F/\p$ is isomorphic to a common
finite field~$\F_{q_0}$.
For all~$\p\in S$, we fix an isomorphism~$\iota_\p\colon \order/\p\order \cong
\mat_d(\F_{q_0})$.

 Let~$\Alpha = \mat_d(\F_{q_0})$ so that~$q = q_0^d$, $s=|S|$ and~$N = ds$, and
define~$\Theta\colon \aritgp \to \Alpha^s$ to be the map sending~$\gamma\in\aritgp$
to the word formed by the~$\iota_\p(\gamma)$ for~$\p\in S$.

 Let us write $n=r_1+2r_2$ and $r_1=u+r$, where $(r_1,r_2)$ is the signature 
 of $F$, $u$ is the number of real places~$\sigma$ that are
 unramified in $A$, and $r$ is the number of real places~$\sigma$ that
 ramify in $A$. Let
 \[
   G =
    \prod_{\sigma\in\PP_\infty}\SL_{d_\sigma}(D_\sigma)\cong  \SL_d(\R)^u
    \times \SL_{d/2}(\Hamil)^r \times \SL_d(\C)^{r_2}.
  \]

Following the notations of Section \ref{section:Cartan}, we define~$\rho\colon G\to \R_{\ge 0}$ componentwise: for all~$g =
(g_\sigma)_{\sigma\in\PP_\infty}\in G$, let
\[
  \rho(g) = \max_{\sigma\in\PP_\infty} \rho(g_\sigma).
\]
For~$t>0$, we define the following compact subset~$\B(t)\subset G$:
\[
  \B(t) = \{g\in G\mid \rho(g)\le t\}.
\]
Note that for all~$g,h\in \B(t)$, we have~$h^{-1}\in \B(t)$ and~$\rho(h^{-1}g) \leq 2t$, by Corollary~\ref{cor:subadd}.

Let~$\Co$ be the code attached to~$(\B(t),\aritgp,\Theta)$ as in
Section~\ref{sec:construction}.
The goal of
this section is to analyse the code~$\Co$, and to obtain aymptotically good
families of codes from this construction.

\subsection{Minimal distance}

Let us start with the following lemma:

\begin{lem}\label{lem:colnorm}
  Let~$\alg$ be a central simple algebra of degree~$d$ over a number field~$F$,
  and let~$\order$ be an order in~$\alg$. Let~$\p$ be a prime ideal of~$\Z_F$ such
  that there is an isomorphism~$\iota_\p\colon\order/\p\order \cong \mat_d(\F_{q})$.
  Let~$x\in\order$, and let~$r$ be the rank of the
  matrix~$\iota_\p(x)$.
  Then~$|\order/(\p\order+x\order)| = q^{d(d-r)}$.
\end{lem}
\begin{proof}
  Let~$m = \iota_\p(x)$. We have~$\order/(\p\order + x\order) \cong
  \mat_d(\F_q)/(m\cdot\mat_d(\F_q))$. Since $\dim_{\F_q} (m\cdot\F_q^d) = r$ by
  definition, we have~$\dim_{\F_q} (m\cdot\mat_d(\F_q)) = dr$, and
  therefore~$\dim_{\F_q}\mat_d(\F_q)/(m\cdot \mat_d(\F_q)) = d(d-r)$, proving the
  result.
\end{proof}

Concerning the minimum sum-rank distance $d_R(\Co)$ of the code $\Co$, we obtain: 
\begin{propo}\label{prop:distmin}
  We have
  $$d_R(\Co) \geq N - nd^2 \frac{\log 2}{\log q} - \frac{2nd^2t}{\log q} \cdot$$
\end{propo}
\begin{proof}
  Let~$x \neq
  y$ be elements of~$\aritgp\cap c\B(t)$ such that
  $d_R(\Theta(x),\Theta(y))=d_R(\Co)$.
  Note that since~$y\in \Gamma$, the matrix~$\iota_\p(y)$ has determinant~$1$
  and is invertible with inverse~$\iota_\p(y)^{-1} = \iota_\p(y^{-1})$.
  Let~$z=y^{-1}x-1\in\order$; the element~$z$ is nonzero and
  therefore~$\N(z)\neq 0$ since $A$ is a division algebra.
  For each~$\p\in S$, let~$r_\p$ be the rank
  of~$\iota_\p(z)$.
  Since multiplying a matrix by the invertible matrix~$\iota_\p(y)^{-1}$ on the left does not change
  the rank, we get~$\displaystyle{\sum_{\p\in S} r_\p = d_R(\Co)}$.
  Moreover, for all~$\p\in S$ 
  we have~$|\order/(\p\order+z\order)|\ge
  q_0^{d(d-r_\p)}$ by Lemma~\ref{lem:colnorm}. We obtain
  \[
    \N(z) = |\order/z\order| \ge \prod_{\p\in S}|\order/(\p\order+z\order)| \ge \prod_{\p\in
    S}q_0^{d(d-r_\p)} = q_0^{d(N-d_R(\Co))} = q^{N-d_R(\Co)}.
  \]
  On the other hand, let us write $x=cx_0$ and $y=cy_0$, with $x_0,y_0 \in
  \B(t)$ and where~$c\in G$ is as in Section~\ref{sec:construction}. Since
  $\rho(y^{-1}x)=\rho(y_0^{-1}x_0) \leq 2t$, by
  Corollary~\ref{cor:nrdbound} we obtain
  \[
    \N(z) = \prod_{\sigma\colon F\hookrightarrow \C}|\sigma(\nrd(z))|^d \le
    2^{nd^2}\exp(nd^2\rho(y^{-1}x))\le 2^{nd^2}\exp(2nd^2t).
  \]
  Taking logarithms and dividing by~$\log q$ gives the result.
  \end{proof}

As consequence, we have:

\begin{cor} \label{coro:bound-t}
Suppose that $t>0$ is such that $\displaystyle{2t \leq \frac{N \log q}{n d^2}
  - \log 2}$. Then $\Theta|_{\aritgp \cap c\B(t)}$ is injective.
\end{cor}

\subsection{Number of codewords}

Recall that $\displaystyle{G= \prod_{\sigma \in \PP_\infty}
\SL_{d_\sigma}(D_\sigma)}$. Then $G$ inherits the product topology and the
product measure $\displaystyle{\otimes_\sigma d\mu_\sigma}$, where the volume
form $d\mu_\sigma$ is normalised  as in Section \ref{section:Haar-measure}.
We start with Prasad's formula for the  volume~$G/\order^1$.
 
\begin{propo}\label{prop:volPrasad}
  We have
  \[
    \mu(G/\order^1) =
    d^{\frac{n}{2}}
    \left(\frac{\Delta_A}{\Delta_F}\right)^{1/2}
    \prod_{j=2}^d\zeta_F(j)\cdot \Phi,
  \]
  where
  \[
    \Phi = \prod_{\p}\prod_{0<i<d, e_\p\nmid i}(1-N(\p)^{-i}).
  \]
\end{propo}

\begin{proof}
  We use Prasad's formula \cite[Theorem 3.7]{Prasad} where the normalisation of the  volume is
  defined as follows. On each factor~$\lieg_\sigma =
  \liesl_{d_\sigma}(D_\sigma)$, the choice of a
  volume form determines one on~$\lieg_\sigma\otimes_\R\C$. Prasad chooses the
  volume form that gives volume~$1$ to a maximal compact subgroup~$K'\cong
  \SU_d(\C)^n$ of~$\G(\C) = (A\otimes_\R\C)^1$. For this normalisation, Prasad's formula
  yields
  \[
    \mu_{\rm Pras}(G/\order^1) =  \left(\frac{\Delta_A}{\Delta_F}\right)^{1/2}
    \left(\prod_{k=1}^{d-1}\frac{k!}{(2\pi)^{k+1}}\right)^{n}\prod_{j=2}^d\zeta_F(j)\cdot
    \Phi
  \]
  (for details, the reader can refer to \cite[Theorem 2.4.1.10]{Page}).
  To relate our normalisation to the one of Prasad, we
  relate the norm~$\|\cdot\|$ on~$\liesl_{d_\sigma}(D_\sigma)\otimes_\R\C\cong
  \liesl_d(\C)^{[F_\sigma\colon \R]}$ that we defined previously, to the
  norm~$\|\cdot\|_\C$ on the same Lie algebra, induced by the one
  on~$\liesl_{d_\sigma}(D_\sigma)$:
  we find~$\|\cdot\|_\C = \frac{1}{\sqrt{2}}\|\cdot\|$.
  Denoting~$\mu_\C$ the measure induced by the metric~$\|\cdot\|_\C$, we have
  \[
    \mu_\C(K') =
    2^{-\frac{n(d^2-1)}{2}} \mu(\SU_d(\C))^n
    = 2^{-\frac{n(d^2-1)}{2}}
    \left(2^{d^2+d/2-3/2}\sqrt{d}\prod_{k=1}^{d-1}\frac{\pi^{k+1}}{k!}\right)^{n}
  \]
  by Proposition~\ref{prop:volK}(\ref{item:volsud}).
  Finally, we get the formula since
  \[
    \mu(G/\order^1) = \mu_{\rm Pras}(G/\order^1)\mu_{\C}(K').
  \]
\end{proof}

\begin{cor}\label{cor:volPrasad}
  We have
  \[
    \log \mu(G/\order^1) \le 
    \frac{1}{2}\log\left(\frac{\Delta_A}{\Delta_F}\right) + O(n\log d).
  \]
\end{cor}
\begin{proof}
  It follows from~$\Phi\le 1$ and~$\zeta_F(j)\le \zeta(j)^{n}\le (1+O(2^{-j}))^n$
  for~$j\ge 2$.
\end{proof}

This estimate allows us to obtain the following bound on the rate of $\Co$:

\begin{propo}\label{prop:sizecode} 
  Assume that~$s=|S|=n$.
  For all~$t \geq 1$  as in Corollary \ref{coro:bound-t},  we have
  \[
    \frac{1}{N}\log|\Co| \ge \frac{d}{200}t - \frac{d^2-1}{2d}\log\rd_F -
    \frac{1}{2n}\log \N(\delta_A) - \frac{3d}{2}\log d + O(d).
  \]
\end{propo}

\begin{proof}
  We have~$N=dn$, and by Lemma~\ref{lem:boundratevol}, $|\Co| \geq  \mu(\B(t)) / \mu(G/\order^1)$.
  Since~$\B(t)$ is a product of balls on each of the factors,
  Proposition~\ref{prop:volball},
  together with Corollary~\ref{cor:volPrasad} and the relation~$\Delta_A =
  N(\delta_A)^d\Delta_F^{d^2}$, gives the result.
\end{proof}

\subsection{Analysis of the code} \label{section:analysis-multiplicative}

Let us start with the existence of unramified towers of number fields with splitting conditions.

\begin{propo}\label{prop:goodnf}
  There exist an integer~$M_0$ and a real number~$C>0$ such that for all
  primes~$p$ satisfying~$\left(\frac{M_0}{p}\right)=1$, there exists a sequence
  of number fields~$(F_k)_k$ such
  that: $(i)$ $ [F_k\colon \Q]\to\infty$, $(ii)$  the prime $p$ splits totally
  in~$F_k/\Q$, and $(iii)$~$\rd_{F_k}\le C$.
\end{propo}
\begin{proof}
  This follows from well-known methods in the study of towers of bounded root
  discriminant, using a tower above a quadratic field. The reader may refer to \cite[Section~5]{Maire-Oggier}.
\end{proof}

We now prove the main result of this work.

\begin{thm} \label{theorem:main}
  For all~$d\ge 2$, there exists a family of asymptotically good number field
  codes for the sum-rank distance,
  each obtained from the group of units of reduced norm~$1$ in a maximal order
  in a division algebra of degree~$d$,
  over a fixed
  alphabet~$\mat_d(\F_p)$, where~$\log p = c\log d + O(\log\log d)$
  and~$c>0$ is a constant.
\end{thm}
\begin{proof}
  Let~$d\ge 2$. We pick a family of number fields~$F_k$ with~$\rd_{F_k}\le C$ as in
  Proposition~\ref{prop:goodnf}, leaving~$p\ge 5$ to be chosen later. Fix $F_k$
  and let $n=[F_k:\Q]$.
  We choose~$A$ a central division algebra of degree~$d$ over~$F_k$ ramified
  exactly at one prime~$\p_2$ above~$2$ and one prime~$\p_3$ above~$3$; by Class
  Field Theory such an algebra does exist (see for instance~\cite[Section 32,
  especially (32.13)]{Reiner} and the references therein).
  This implies~$\N(\delta_A)^{\frac{1}{nd}}\le 6$.
  We choose~$S$ to be the set of primes of~$F_k$ above~$p$, so that~$q_0=p$,
  $q=p^d$, $s=n=[\F_k:\Q]$ and~$N = nd$. Let  $\Co$ be the unit group code constructed from $A$.
  The quantity of Proposition~\ref{prop:sizecode}
  \[
    \frac{d}{200}t - \frac{d^2-1}{2d}\log\rd_F -
    \frac{1}{2n}\log \N(\delta_A) - \frac{3d}{2}\log d + O(d)
  \]
  can be written as
  \[
    \frac{d}{200}t - \frac{3d}{2} \log d + f(d) \text{ where }f(d) = O(d)\cdot
  \]
  We pick~$t\ge 1$ minimal such that~$\displaystyle{\frac{d}{200}t - \frac{3d}{2}\log d
  + f(d)\ge 1}$, so that~$t = 300\log d +
  O(1)$.
   By Proposition~\ref{prop:distmin} we have
  \[
    \frac{d_R(\Co)}{N} \geq 1 - \frac{\log 2}{\log p} - \frac{2t}{\log p}\cdot
  \]
  We pick~$p$ such that~$\displaystyle{\log p \ge \frac{1}{1-d^{-1}}(\log 2 + 2t)}$, so
  that~$\displaystyle{\frac{d_R(\Co)}{N}\ge \frac{1}{d}}$.
  By Proposition~\ref{prop:sizecode} and by the choice of $t$,  we have
  \[
    \frac{1}{N}\log|\Co| \ge 1.
  \]
Since any~$p\ge 5$ such that~$(\frac{M_0}{p})=1$ can be used, by the Dirichlet arithmetic
  progression theorem, there exists such~$p$ with~$\log p = 2t + O(\log t)$,
  i.e.~$\log p = 600\log d + O(\log\log d)$.
\end{proof}

\begin{rmk} The main contributions come from the volume of~$K$ and from
  the exponential growth rate of the volume of~$\B(t)$.
  We obtain~$c=600$ as an admissible value. We did not try to optimise this
  constant. It would be interesting to find families of unit group codes with a
  better asymptotic behaviour of~$\log p$ as~$d\to\infty$.
\end{rmk}

\subsection{Quaternion case}\label{sec:quaternion}

In this section, $A$ is a quaternion algebra, i.e.~$d=2$.
First, it is easy to give a closed formula for the volume of $\B(t)$.
\begin{propo} We have
  \[
    \mu(\B(t)) =
    2^{\frac{3}{2}u+\frac{5}{2}r+4r_2}\pi^{2r_1+3r_2}(\cosh(2t)-1)^{u}(\sinh(4t)-4t)^{r_2}.
  \]
\end{propo}
\begin{proof}
  We compute the volume on each factor. 
  
  By Proposition~\ref{prop:volK}, we
  have~$\mu(\SO_2(\R)) = 2\sqrt{2}\pi$, $\mu(\SU_2(\C)) = 16\pi^2$
  and~$\mu(\SL_1(\Hamil)) = 4\sqrt{2}\pi^2$. In the~$\SL_2(\R)$ case, the integral to
  compute is
  \[
    \int_0^t\sinh(2a)da = \frac{1}{2}(\cosh(2t)-1).
  \]
  In the~$\SL_2(\C)$ case, the integral involved is
  \[
    \int_0^t\sinh(2a)^2da = \frac{1}{8}(\sinh(4t)-4t).
  \]
  Putting these together gives the result.
\end{proof}

In the quaternion case, Prasad's formula becomes:

\begin{propo}
  We have
  \[
    \mu(G/\order^1) = 2^{n/2} (\Delta_F)^{\frac{3}{2}}
    \zeta_F(2) \prod_{\p\mid\delta_A}(N(\p)-1).
  \]
\end{propo}

\medskip

To conclude this section, let us give an explicit example.
  Let~$F = \Q(\cos(2\pi/11),\sqrt{2},\sqrt{-23})$ (from~\cite{Martinet}). The $2$-class group of~$F$ is
  isomorphic to~$(\Z/2\Z)^9$. Let~$F_1$ be the $2$-Hilbert class field of~$F$.
  Let $p$ be a prime number  such that  every prime ideal above $p$ in~$F$
  splits completely in~$F_1/F$; take a such prime ideal $\p$.
   Then there exists an unramified infinite extension~$L/F$
  such that~$\p$ splits completely in~$L$ (see \cite[Example 9]{Maire-Oggier}). Let~$F_1\subset F_k\subset L$ be an
  intermediate field, of degree~$n$ over $\Q$. 
  By construction there exists~$n/20$ primes of~$F_k$ above~$\p$ of residue degree~$f_p$, where $f_pg_p=20$.
  Let~$A$ be the quaternion
  algebra over $F_k$ ramified at exactly two of these primes, and let~$S$ be the set of the
  remaining ones, which has size~$\displaystyle{s = \frac{n}{20}-2}$; $N=ds$. Let~$\Co$ be the unit
  group code constructed from~$A$. Then
  \[
    \frac{\mu(\B(t))}{\mu(G/\order^1)} =
    \frac{\big(2^{2}\pi^{3/2}(\sinh(4t)-4t)^{1/2}\big)^{n}}{2^{n/2}\rd_{F_k}^{3n/2}\zeta_{F_k}(2)(p^{f_p}-1)^2},
    \]
    in particular
    \[ \Big(\frac{\mu(\B(t))}{\mu(G/\order^1)}\Big)^{1/n}
    \ge 
    \frac{(2\pi)^{3/2}(\sinh(4t)-4t)^{1/2}}{\rd_F^{3/2}\zeta_{F}(2)^{1/20}(p^{f_p}-1)^{2/n}}\cdot
  \]
  We have~$\rd_F \approx 92.37$ and~$\zeta_F(2)\approx 1.02$.
  For~$t = 2.2$ we have
  \[
    \frac{(2\pi)^{3/2}(\sinh(4t)-4t)^{1/2}}{\rd_F^{3/2}\zeta_{F}(2)^{1/20}}>1,
  \]
  so this is an admissible value for an asymptotically good code. From the
  minimal distance formula, we can choose any~$p$ such that~$q_0=p^{f_p}$ satisfies
  \[
    \log q_0 > \frac{n}{s}(\log 2 +2t).
  \]
  Asymptotically we can take any~$p$ such that~$\displaystyle{\frac{1}{20}\log q_0 > 5.09}$,
  i.e.~$p^{f_p/20} \geq 163$ and therefore~$q = p^{df_p} \geq 163^{40}$.

\section{Complement: the additive construction}\label{sec:additiveconstr}

\subsection{The construction (following \cite{Maire-Oggier})}
Let us recall the additive construction. 
Let $F$ be a number field of degree $n$ over $\Q$, and 
let $A$ be a central division algebra of degree~$d$ over~$F$.
 Consider the locally compact group $G=\prod_{\sigma \in \PP_\infty} \mat_{d_\sigma}(D_\sigma)$, equipped with
\begin{enumerate}[$\bullet$]
\item the Euclidean norm $\displaystyle{\T_2(g)=\sum_{\sigma \in \PP_\infty}
  n_\sigma \|\sigma(g)\|_2^2}$, where $\|(m_{i,j})\|_2=\sqrt{\sum_{i,j} e_\sigma |m_{i,j}|^2}$;
\item the Lebesque measure $d\mu$ relative to an orthonormal basis of~$G$ with
  respect to~$\T_2$.
\end{enumerate}

\smallskip

Let~$\G=A$ be the algebraic additive group, and~$\aritgp = \order$, where~$\order$ is a maximal order of~$A$. 

\smallskip
As for the multiplicative group code, let~$S$ be a
finite set of ideal primes of~$\Z_F$ that are unramified in~$\alg$  and such
that for all~$\p\in S$, the residue field~$\Z_F/\p$ is isomorphic to a common
finite field~$\F_{q_0}$. For all~$\p\in S$, we fix an isomorphism~$\iota_\p\colon \order/\p\order \cong
\mat_d(\F_{q_0})$.

\smallskip

 Let~$\Alpha = \mat_d(\F_{q_0})$ so that~$q = q_0^d$, $s=|S|$ and~$N = ds$, and
define~$\Theta\colon \aritgp \to \Alpha^s$ to be the map sending~$\gamma\in\aritgp$
to the word formed by the~$\iota_\p(\gamma)$ for~$\p\in S$.

\smallskip

Take the ball $\B(t)=\{(x_\sigma)_\sigma \in \G \mid \n{x_\sigma} \leq t\}$.
Let~$\Co$ be the code attached to~$(\B(t),\aritgp,\Theta)$ as in
Section~\ref{sec:construction}.

\begin{propo}\label{prop:minidistanceadditive}
We have 
\[
  d_R(\Co) \geq N- d^2n \frac{\log(2t)}{\log q} +\frac{d^2 n}{2}\cdot \frac{\log
  d}{\log q}.
\]
\end{propo}

\begin{proof}
We follow the proof of Proposition \ref{prop:distmin}. 
  Let~$x \neq y$ in $\aritgp\cap (c+\B(t))$ be such that
  $d_R(\Theta(x),\Theta(y))=d_R(\Co)$,
  where~$c\in G$ is as in Section~\ref{sec:construction}.
  Let~$z=x-y\in\order$; the element $z$ is nonzero and therefore~$\N(z)\neq 0$.
  As for the multiplicative case, we obtain
  $\displaystyle{\N(z) \ge  q^{N-d_R(\Co)}}$.
  On the other hand, let us write $x=c+x_0$ and $y=c+ y_0$, with $x_0,y_0 \in
  \B(t)$. As $\n{\sigma(x-y)}  \leq 2t$, and $\T_2(z) \leq n (2t)^2$, then we obtain 
  \[
    \N(z) = \prod_{\sigma : F \hookrightarrow \C}|\sigma(\nrd(z))|^d 
   \le \left(\frac{2t}{d^{1/2}}\right)^{d^2n},
  \]
  thanks to the estimate (see \cite[Chapter 2, \S3]{Page}): $|N_{F/\Q} \nrd(x)|^d \leq \left(\frac{\T_2(x)}{dn}\right)^{d^2n/2}$.
  Taking logarithms and dividing by~$\log q$ gives the result.
\end{proof}

Concerning the number of codewords, we have:
 \begin{propo} \label{prop:volumecodeadditif} Suppose that $\Theta|_{\aritgp\cap
   (c+\B(t))}$ is injective. Then we have
 \[
 |\Co| \geq \frac{2^{r_2d^2} \ \V_{d^2}^{r_1} \ \V_{2d^2}^{r_2} \ t^{d^2n}}{\sqrt{\disc_A}},
 \]
 where $\V_n$ denotes the volume of the unit  ball of the space $\R^n$ equipped with the Lebesgue volume form.
\end{propo} 
 
\begin{proof}
  Here $\mu(\B(t))= 2^{r_2d^2}\V_{d^2}^{r_1} \V_{2d^2}^{r_2} t^{d^2n}$, and
  $\mu(G/\order)= \sqrt{\disc_A}$ \footnote[2]{A factor~$2^{-r_2d^2}$
  is missing in~\cite[Proposition 9]{Maire-Oggier}, but this does not affect the results of the
  paper.}.
  Then apply Lemma~\ref{lem:boundratevol}.
\end{proof}

\begin{cor}\label{cor:volumeadditif}
  Suppose that $\Theta|_{\aritgp\cap (c+\B(t))}$ is injective. Then we have
  \[
    \log |\Co| \geq nd^2\log t - \frac{1}{2}\log\disc_A - nd^2\log d + O(nd^2).
  \]
\end{cor}
\begin{proof}
  Use~$\V_n = -\frac{n}{2}\log n + O(n)$.
\end{proof}

\subsection{Asymptotic analysis}

We now follow the case of the multiplicative group of Section~\ref{section:analysis-multiplicative} to obtain the following

\begin{thm} \label{theorem:asymptotic-additif} 
  For all~$d\ge 2$, there exists a family of asymptotically good number field
  codes for the sum-rank distance, each obtained from the additive group of a
  maximal order in a division algebra of degree~$d$, over a fixed
  alphabet~$\mat_d(\F_p)$, where~$\log p = \frac{1}{2}\log d + O(\log\log d)$.
\end{thm}

\begin{proof}
Let~$d\ge 2$. As in the proof of Theorem \ref{theorem:main}, we pick a family of
  number fields~$F_k$, a prime number $p$, a central division algebra~$A$
  ramified exactly at one prime $\p_2$ above~$2$ and at one prime $\p_3$
  above~$3$, and a set $S$ of primes, so that~$q_0=p$,
  $q=p^d$, $s=n=[\F_k:\Q]$ and~$N = nd$. Let  $\Co$ be the additive group code
  constructed from $A$ and $S$.
  Then by Corollary~\ref{cor:volumeadditif} we have
  \[
    \frac{1}{N} \log |\Co|
    \ge d\log t - \frac{1}{2nd}\log\disc_A - d\log d + O(d)
    = d\log t - d\log d + f(d) \text{ where }f(d) = O(d).
  \]
  We pick~$t\ge 1$ minimal such that~$ -d \log d +  d \log t + f(d) \geq 1$, so that 
  \[
    \log t = \log d + O(1).
\]
   By Proposition~\ref{prop:minidistanceadditive} we have
  \[
    \frac{d_R(\Co)}{N} \geq 1 - \frac{\log 2t }{\log p} + \frac{\log d }{2\log p}\cdot
  \]
 We pick~$p$ such that~$\displaystyle{\log p \ge \frac{1}{1-d^{-1}}\left(\frac{1}{2}\log d +O(1) \right)}$, so
  that~$\displaystyle{\frac{d_R(\Co)}{N}\ge \frac{1}{d}}$. 
  By Proposition~\ref{prop:volumecodeadditif} (and by the choice of $t$),  we have $\displaystyle{\frac{1}{N}\log|\Co| \ge 1}$.
  As before, by the Dirichlet arithmetic
  progression theorem, there exists such~$p$ with~$\log p =\frac{1}{2}\log d + O (\log \log d) $.
\end{proof}

\subsection{Codes over finite fields}

In this section, we explain how to construct codes from quaternion orders, with
alphabet naturally given as a finite field. The difference with the previous
constructions is that the primes in~$S$ need to ramify in~$A$.

\begin{thm} \label{theorem:small-alphabet-quaternion}
  Let  $M_0$ and $C$ as in Proposition~\ref{prop:goodnf}, and let
  $\alpha=-\frac{1}{4} \log \frac{\pi^4}{24}$.
  For all prime~$p$ such that~$\left(\frac{M_0}{p}\right)=1$ and~$\log p > \log
  C + \alpha$, there exists a family of asymptotically good codes over
  $\F_{p^2}$, each obtained from the additive group of a maximal order in a
  quaternion algebra.
\end{thm}
\begin{proof}
  Let~$F_k$ be as in Proposition~\ref{prop:goodnf}.
  Let~$p$ be a prime number that splits totally in~$F_k/\Q$, and let $S$ be a
  maximal subset of primes of $F_k$ above $p$ such that~$|S|$ is even.

  Let~$A$  be a central division algebra of degree~$d$ over $F_k$  ramified exactly at each
  prime ideal~$\p$ of~$S$ and with common ramification index $e>1$.
  Consider the additive codes as in the previous section.
  Writing $d=ef$ and~$\Prm$ the maximal two-sided ideal of~$\order$ above~$\p$,
  then there is an isomorphism~$\iota_\p \colon \order/\Prm \cong \mat_f(\F_{p^e})$, and we use
  columns $\F_{p^e}^f$ as codewords via~$\iota_\p$. We have~$N=f\cdot|S|\ge f\cdot ([F_k:\Q]-1)$, and  $q=p^d$ as
  before. The only difference with
  the unramified case concerns the quantity  $(\disc_A)^{1/([F_k:\Q]d)}$ which
  is not bounded along~$F_k/\Q$. Put $n=[F_k:\Q]$.
  We have
  \[ 
    \frac{1}{dn} \log \disc_A = d \log \rd_F + \frac{|S|}{n}f(e-1) \log p .
  \]
  Then, a good parameter $t>0$ does exist when 
  \[
    \log p > \left(e-\frac{1}{2}\right)\log d + \frac{1}{2}(e-1)\log p + O(e),
  \]
  \emph{i.e} when~$e<3$.
  Although there is no room when $e\geq 3$ (as noted in \cite[\S
  7.5.3]{Maire-Oggier}), we may  construct maximal orders good codes over finite
  fields by using quaternion algebras.
\end{proof}

\begin{rmk} 
  The existence of unramified towers with splitting conditions and \emph{small root discriminant} is central in the asymptotic
  analysis of number field codes, especially when we look for   codes  over $\F_{p^2}$ with  $p$ as small as possible 
  (see for example \cite{Guruswami}). For towers of number fields of small root
  discriminant see for example  \cite{HM-JSC}.
\end{rmk}

\begin{rmk} When $d$ is \emph{even}, our last computation shows how to construct asymptotically good additive  
 codes for the sum-rank distance from  algebras of   degree $d$ with
  alphabet~$\mat_{d/2}(\F_{p^2})$.
\end{rmk}


\end{document}